\documentclass[a4paper,oneside,10pt]{article}%
\usepackage{amsmath}
\usepackage{amsfonts}
\usepackage{amssymb}
\usepackage{graphicx}
\usepackage{enumerate}
\usepackage{color}
\usepackage[square,numbers,sort&compress]{natbib}%
\usepackage{hyperref}
\usepackage[amsmath,thmmarks]{ntheorem}
\hypersetup{hypertex=true,bookmarks=true,bookmarksnumbered	=true,linktocpage=false,colorlinks=true,linkcolor=blue,anchorcolor=blue,citecolor=red}
\usepackage{indentfirst} 
\providecommand{\U}[1]{\protect \rule{.1in}{.1in}}

\newtheorem{theorem}{Theorem}[section]

\newtheorem{corollary}[theorem]{Corollary}

\newtheorem{definition}[theorem]{Definition}

\newtheorem{lemma}[theorem]{Lemma}

\newtheorem{remark}[theorem]{Remark}

\newenvironment{proof}[1][Proof]{\noindent \textbf{#1.} }{\  $\Box$}
\numberwithin{equation}{section}

\begin{document}

\makeatletter
\newcommand{\rmnum}[1]{\romannumeral #1}
\newcommand{\Rmnum}[1]{\expandafter\@slowromancap\romannumeral #1@}
\makeatother
{\theoremstyle{nonumberplain}
}

\title{$ G $-Bessel processes and related properties}
\author{Mingshang Hu \thanks{Zhongtai Securities Institute for Financial Studies,
Shandong University, Jinan, Shandong 250100, PR China. humingshang@sdu.edu.cn.
Research supported by the National Natural Science Foundation of China (No. 12326603, 11671231) and the National Key R\&D Program of China (No. 2018YFA0703900). }
\and Renxing Li\thanks{Zhongtai Securities Institute for Financial Studies,
Shandong University, Jinan, Shandong 250100, PR China. 202011963@mail.sdu.edu.cn.
 }
 }
\maketitle

\textbf{Abstract}. In this paper, we introduce $ G $-Bessel processes for a class of $ d $-dimensional $ G $-Brownian motions. Under the condition of dimensionality $ d $, we obtain that the $ G $-Bessel process is the solution of the stochastic differential equation. Furthermore, under the stricter condition of dimensionality, we establish the existence and uniqueness of a solution of the stochastic differential equation governing the $ G $-Bessel process and prove the nonattainability of the origin for $ G $-Brownian motion.



\addcontentsline{toc}{section}{\hspace*{1.8em}Abstract}

\section{Introduction}

Motivated by the uncertainty problem in the financial market, Peng introduced $ G $-normal distribution, $ G $-expectation and $ G $-Brownian motion via G-heat equation and established the corresponding stochastic differential equation theory (see \cite{peng2004filtration,peng2005nonlinear,peng2007G,peng2008multi,peng2019nonlinear}). In recent years, there has been a lot of research on the stochastic differential equation driven by $ G $-Brownian motion (see \cite{gao2009pathwise,lin2013some,liny2013stochastic,ren2011} and the references therein). Some properties of the stopping time under $ G $-expectation are established in \cite{song2011hitting,liu2020exit}. Sample path properties of $ G $-Brownian motion are investigated in \cite{wang2018sample}.

If $ G $ is a linear function, then the $ G $-expectation reduces to a linear expectation, and the $ G $-Brownian motion simplifies to the classical Brownian motion. Within the framework of linear expectation, Bessel processes have been extensively studied and applied in financial mathematics, see \cite{G2003survey,brownian2012,countinuous2013,brownian2016,yor2001expon} and the references therein. However, as we know, there are no results on Bessel processes under the $ G $-expectation framework ($ G $-Bessel processes). The main purpose of this paper is to investigate $ G $-Bessel processes. 

In deriving the expression for the $ G $-Bessel process, we encountered two major difficulties. First, the quadratic variation process of the $ G $-Brownian motion is a continuous process with independent and stationary increments, which is non-deterministic unless $ G $ is a linear function. This makes it difficult to handle the integral terms related to the quadratic variation process in our calculations. Second, in the classical case, the probability of Brownian motion staying in a ball can be computed using the density function, and the dominated convergence theorem can be applied to obtain the integral limit of the approximation function. However, these tools are not applicable within the framework of $ G $-expectation. 

To overcome these difficulties, we propose a reasonable condition (A) to the function $ G $ (see Section \ref{s3}). In this case, the quadratic variation process satisfies (H) (see Section \ref{s3}). Then by applying L\'{e}vy's martingale characterization of $ G $-Brownian motion, we obtain the equation (\ref{formally}) that the $ G $-Bessel process formally satisfies. Using PDE techniques, we obtain the capacity estimate for the $ G $-Brownian motion staying in a ball, which is better than the capacity estimate in \cite{hu2016quasi}. In the truncated approximation, by using the capacity estimate, we prove that the $ G $-Bessel process is a solution to equation (\ref{formally}). Furthermore, we establish the existence and uniqueness of solution of the equation (\ref{formally}). Based on this conclusion, we obtain that the nonattainability of the origin for the $ G $-Brownian motion.

This paper is organized as follows. In Section \ref{s2}, we recall some necessary notations and results that we will use in this paper. $ G $-Bessel processes is introduced in Section \ref{s3}. In Section \ref{s4}, we obtain some properties of $ G $-Bessel processes. 

\section{Preliminaries}\label{s2}

In this section, we recall some results in the $ G $-expectation framework. More relevant details can be found in \cite{peng2004filtration,hu2016quasi,hu2021on,peng2005nonlinear,peng2007G,peng2008multi,peng2019nonlinear,denis2011function}.

Let $ \Omega $  be a given nonempty set and let $ \mathcal{H} $ be a linear space of real-valued functions defined on $ \Omega $ such that $ 1\in\mathcal{H} $ and $ \left|X \right|\in\mathcal{H}  $ for any $ X\in\mathcal{H} $. Moreover, we suppose that if $ X_{1},\cdots,X_{n}\in\mathcal{H} $, then $ \varphi(X_{1},\cdots,X_{n})\in\mathcal{H} $ for each $ \varphi\in C_{b,Lip}(\mathbb{R}^{n}) $, where $ C_{b,Lip}(\mathbb{R}^{n}) $ denotes the linear space of all bounded Lipschitz functions on $ \mathbb{R}^{n} $. The space $ \mathcal{H} $ is regarded as the space of random variables. An $ n $-dimensional random vector $ X=(X_{1},\cdots,X_{n}),$ where $ X_{i}\in\mathcal{H} $, is denoted by $ X\in\mathcal{H}^{n} $. We denote the scalar product of $ x,y\in\mathbb{R}^{n} $ by $ \langle x,y\rangle  $ . The Euclidean norm of $ x $ is defined as $ |x|=\langle x,x\rangle^{\frac{1}{2}}  $.

\begin{definition}
	A sublinear expectation is a functional $ \hat{\mathbb{E}}:\mathcal{H}\rightarrow\mathbb{R} $ satisfying the following properties: for each $ X,Y\in\mathcal{H} $, we have the following.
	\begin{enumerate}[(i)]
		\item\label{1} Monotonicity: $ X\geq Y $ implies $\hat{\mathbb{E}}[X]\geq \hat{\mathbb{E}}[Y]  $;
		\item\label{2} Constant preserving: $ \hat{\mathbb{E}}[c]=c $ for $ c\in \mathbb{R} $;
		\item\label{3} Sub-additivity: $ \hat{\mathbb{E}}[X+Y]\leq\hat{\mathbb{E}}[X]+\hat{\mathbb{E}}[Y] $;
		\item\label{4} Positive homogeneity: $ \hat{\mathbb{E}}[\lambda X]=\lambda\hat{\mathbb{E}}[X] $ for $ \lambda>0 $.
	\end{enumerate}
\end{definition}

The triple $ (\Omega,\mathcal{H},\hat{\mathbb{E}}) $ is called a sublinear expectation space. 

\begin{remark}
	If the inequality in $ (3) $ takes equality, then $ \hat{\mathbb{E}} $ reduces to a linear expectation and $ (\Omega,\mathcal{H},\hat{\mathbb{E}}) $ becomes a linear expectation space.
\end{remark}

\begin{definition}
	Two $ n $-dimensional random vectors $ X_{1} $ and $ X_{2} $, defined respectively on sublinear expectation spaces $(\Omega_{1},\mathcal{H}_{1},\hat{\mathbb{E}}_{1})  $ and $ (\Omega_{2},\mathcal{H}_{2},\hat{\mathbb{E}}_{2}) $, are said to be identically distributed, denoted by $ X_{1}\overset{d}{=}X_{2} $, if 
	\[\hat{\mathbb{E}}_{1}[\varphi(X_{1})]=\hat{\mathbb{E}}_{2}[\varphi(X_{2})], \quad \forall \varphi\in C_{b,Lip}(\mathbb{R}^{n}).  \]
\end{definition}

\begin{definition}
	On the sublinear expectation space $ (\Omega $,$\mathcal{H},\hat{\mathbb{E}})$, an $ n $-dimensional random vector $ Y $ is said to be independent from an $ m $-dimensional random vector $ X $ if 
	\[\hat{\mathbb{E}}[\varphi(X,Y)]=\hat{\mathbb{E}}[\hat{\mathbb{E}}[\varphi(x,Y)]_{x=X}], \quad \forall \varphi\in C_{b,Lip}(\mathbb{R}^{m+n}). \]
\end{definition}	

Let $ G:\mathbb{S}_{d}\rightarrow\mathbb{R} $ be a given monotonic, sublinear mapping. That is, for $ A,B\in\mathbb{S}_{d} $, 
\begin{align*}
	\begin{cases}
		G(A)\leq G(B)  \text{ if } A\leq B,\\
		G(A+B)\leq G(A)+G(B),\\
		G(\lambda A)=\lambda G(A) \text{ for }\lambda\geq0,
	\end{cases}
\end{align*}
where $ \mathbb{S}_{d} $ is the collection of $ d\times d $ symmetric matrices. According to \cite{peng2019nonlinear}, there exists a bounded, convex and closed subset $ \Gamma\subset\mathbb{S}_{d}^{+} $ such that
\begin{align}\label{G}
	 G(A)=\frac{1}{2}\sup_{\gamma\in \Gamma}\mathrm{tr}(\gamma A),
\end{align} 
where $ \mathbb{S}_{d}^{+} $ is the collection of non-negative elements in $ \mathbb{S}_{d} $. In this paper, we always assume that there exist a constant $ 0<\underline{\sigma}<\infty  $ such that
\[\frac{1}{2}\underline{\sigma}^{2}\mathrm{tr}[A-B]\leq G(A)-G(B), \quad \text{for } A\geq B. \]

\begin{definition}[$ G $-normal distribution]
	On the sublinear expectation space $ (\Omega $,$\mathcal{H},\hat{\mathbb{E}})$, a $ d $-dimensional random vector $ \xi $ is called $ G $-normally distributed if for each $ \varphi\in C_{b,Lip}(\mathbb{R}^{d}) $, 
	\[u(t,x):=\hat{\mathbb{E}}[\varphi(x+\sqrt{t}\xi)],\]
	is the viscosity solution to the following $ G $-heat equation:
	\begin{align*}
		\begin{cases}
			\partial_{t}u(t,x)-G(D^{2}_{x}u(t,x))=0,\\
			u(0,x)=\varphi(x),
		\end{cases}
	\end{align*}
	where $ D^{2}_{x}:=(\partial^{2}_{x_{i}x_{j}})_{i,j=1}^{d} $.  	
\end{definition}

\begin{definition}[$ G $-Brownian motion starting at $ x $]
	Let $ (\Omega $,$\mathcal{H},\hat{\mathbb{E}})$ be a sublinear expectation space and let $ x\in\mathbb{R}^{d} $. A $ d $-dimensional process $ B^{x}=\{B_{t}^{x},\ t\geq0 \} $ is called a $ G $-Brownian motion starting at $ x $ if the following properties are satisfied:
	\begin{enumerate}[(i)]
		\item $ B_{0}^{x}=x $;
		\item For each $ t,s\geq 0 $, $ B_{t+s}^{x}-B_{t}^{x} $ is independent from $ (B_{t_{1}}^{x},B_{t_{2}}^{x},\cdots,B_{t_{n}}^{x}) $ for $ n\in\mathbb{N} $ and $ 0\leq t_{1}\leq\cdots\leq t_{n}\leq t $;
		\item $ B_{t+s}^{x}-B_{t}^{x}\overset{d}{=}\sqrt{s}\xi $ for $ t,s\geq0 $, where $ \xi $ is $ G $-normally distributed.
	\end{enumerate}
    If $ x=0 $, we denote $ B^{x} $ simply as $ B $ and refer to $ B $ as a $ d$-dimensional $ G $-Brownian motion.
\end{definition}

\begin{remark}
	If $ x=0 $ and $ G(A)=\frac{1}{2}\mathrm{tr}[A] $ for all $ A\in\mathbb{S}_{d} $, then $ (B_{t})_{t\geq0} $ is a classical Brownian motion.  
\end{remark}

Let $ B $ be a $ d $-dimensional $ G $-Brownian motion on $ (\Omega $,$\mathcal{H},\hat{\mathbb{E}})$. For each $ t\geq0 $, we define the following spaces:

\begin{itemize}
	\item $ Lip(\Omega_{t}):=\{\varphi(B_{t_{1}},\cdots,B_{t_{n}}):n\in\mathbb{N}, 0\leq t_{1}\leq\cdots\leq t_{n}\leq t,\varphi\in C_{b,Lip}(\mathbb{R}^{d\times n})\} $;
	\item $ Lip(\Omega):=\cup_{n=1}^{\infty}Lip(\Omega_{n}) $;
	\item $ \mathcal{F}_{t}:=\sigma(B_{s}:s\leq t) $;
	\item $ \mathcal{F}:=\bigvee_{t\geq0}\mathcal{F}_{t} $.
\end{itemize}

Define $ \|X\|_{L_{G}^{p}}=(\hat{\mathbb{E}}[|X|^{p}])^{\frac{1}{p}}$ for $ X\in Lip(\Omega) $ and $ p\geq1 $. We denote by $ L_{G}^{p}(\Omega) $ the completion of $ Lip(\Omega) $ under the norm $ \|\cdot\|_{L_{G}^{p}} $. The space $  L_{G}^{p}(\Omega) $ is a Banach space and $ \hat{\mathbb{E}}:Lip(\Omega)\rightarrow\mathbb{R} $ can be continuously extended to the mapping from $ L_{G}^{p}(\Omega) $ to $ \mathbb{R} $. The mapping $ \hat{\mathbb{E}}:L_{G}^{p}(\Omega)\rightarrow\mathbb{R} $ is referred to as the $ G $-expectation. Analogously, $ L_{G}^{p}(\Omega_{t}) $ can be defined for each fixed $ t \geq 0 $. The sublinear expectation space $(\Omega ,L_{G}^{1}(\Omega),\hat{\mathbb{E}})  $ is called $ G $-expectation space. 

Moreover, the conditional expectation of $ \xi=\varphi(B_{t_{1}},B_{t_{2}}-B_{t_{1}},\cdots,B_{t_{n}}-B_{t_{n-1}}) \in Lip(\Omega) $ under the $ G $-expectation space can be defined by
\[ \hat{\mathbb{E}}_{t_{i}}[\xi]=\phi(B_{t_{1}},\cdots,B_{t_{i}}-B_{t_{i-1}}), \]
where $ \phi(x_{1},\cdots,x_{i})=\hat{\mathbb{E}}[\varphi(x_{1},\cdots,x_{i},B_{t_{i+1}}-B_{t_{i}},\cdots,B_{t_{n}}-B_{t_{n-1}})],\ 0\leq i\leq n $. Analogously, for each fixed $ t \geq 0 $, $ \hat{\mathbb{E}}_{t}:Lip(\Omega)\rightarrow Lip(\Omega_{t}) $ can be continuously extended to the mapping from $ L_{G}^{p}(\Omega) $ to $ L_{G}^{p}(\Omega_{t}) $.  

On the  $ G $-expectation space, we have the following definition.

\begin{definition}
	Let $ Y\in (L_{G}^{1}(\Omega))^{d} $ be a $ d $-dimensional random vector. For each given $ t\geq0 $, $ Y $ is said to be independent from $ L_{G}^{1}(\Omega_{t}) $ if for any $ X\in L_{G}^{1}(\Omega_{t}) $, it holds that 
	\[\hat{\mathbb{E}}[\varphi(X,Y)]=\hat{\mathbb{E}}[\hat{\mathbb{E}}[\varphi(x,Y)]_{x=X}],\quad \forall\varphi\in C_{b,Lip}(\mathbb{R}^{d+1}). \]
\end{definition}

Let $ \langle B \rangle $ denote the quadratic variation process of $ d $-dimensional $ G $-Brownian motion $ B $. The following theorem can be consulted in \cite{peng2019nonlinear}.

\begin{theorem}\label{quad B}
	For $ t\geq 0 $, 
	\[\langle B \rangle_{t}\in t\Gamma:=\{t\times\gamma:\gamma\in\Gamma\}, \]
	where $ \Gamma $ satisfies (\ref{G}). Moreover, 
	\[\hat{\mathbb{E}}[\varphi(\langle B \rangle_{t})]=\sup_{\gamma\in \Gamma}\varphi(t\gamma), \quad \forall\varphi\in C_{l,Lip}(\mathbb{S}_{d}),\]
	where $ C_{l,Lip}(\mathbb{S}_{d}) $ denotes the linear space of all local Lipschitz functions on $ \mathbb{S}_{d} $.
\end{theorem}

\begin{definition}
	Let $(\Omega ,L_{G}^{1}(\Omega),\hat{\mathbb{E}})  $ be a $ G $-expectation space. A $ d $-dimensional process $ M_{t}\in (L_{G}^{1}(\Omega_{t}))^{d},\ t\geq0 $, is called a martingale if $ \hat{\mathbb{E}}_{s}[M_{t}]=M_{s} $ for each $ s\leq t $. Moreover, if $ d $-dimensional martingale $ M $ satisfies $\hat{\mathbb{E}}[M_{t}]=-\hat{\mathbb{E}}[-M_{t}]  $, then it is called a symmetric martingale.
\end{definition}

We now recall L\'{e}vy's martingale characterization of $ G $-Brownian motion. The readers may refer to \cite{xu2009martingale,xu2010martingale,song2013char,hu2019levy} for more details and the following results are given in \cite{hu2019levy}.

\begin{theorem}[\cite{hu2019levy}]\label{Levy}
	Let $ G $ be a given monotonic and sublinear function and let $(\Omega ,L_{G}^{1}(\Omega),\hat{\mathbb{E}}) $ be a $ G $-expectation space. Assume $ (M_{t})_{t\geq0} $ is a $ d $-dimensional symmetric martingale satisfying $ M_{0}=0 , M_{t}\in(L_{G}^{3}(\Omega_{t}))^{d} $ for each $ t\geq0 $ and $\sup\{\hat{\mathbb{E}}[|M_{t+\epsilon}-M_{t}|^{3}]:t\leq T\}=o(\epsilon)$ as $ \epsilon\downarrow0 $ for each $ T\geq0 $. If the process $ \frac{1}{2}\left\langle AM_{t},M_{t}\right\rangle-G(A)t ,t\geq0 $, is a martingale for each $ A\in\mathbb{S}_{d} $, then $ M $ is a $ G $-Brownian motion.
\end{theorem}

\begin{remark}
	Let  $ d=1 $ and let $ G(x)=\frac{1}{2}(\bar{\sigma}^{2}x^{+}-\underline{\sigma}^{2}x^{-}) $ where $ 0<\underline{\sigma}\leq\bar{\sigma}<\infty $. Let $ (M_{t})_{t\geq0} $ be a symmetric martingale satisfying $ M_{0}=0 , M_{t}\in L_{G}^{3}(\Omega_{t}) $ for each $ t\geq0 $. If $\sup\{\hat{\mathbb{E}}[|M_{t+\epsilon}-M_{t}|^{3}]:t\leq T\}=o(\epsilon)$ as $ \epsilon\downarrow0 $ for each $ T\geq0 $ and $ \{M_{t}^{2}-\bar{\sigma}^{2}t\}_{t\geq0} \ ,  \{-M_{t}^{2}+\underline{\sigma}^{2}t\}_{t\geq0} $ are martingales, then $ M $ is a one-dimensional $ G $-Brownian motion.
\end{remark}

The following is the representation theorem for $ G $-expectation.

\begin{theorem}[\cite{hu2009representation,denis2011function}]\label{thm2.1}
	There exist a weakly compact family $ \mathcal{P} $ of probability measures on $ (\Omega,\mathcal{F}) $ such that
	\begin{align}
		\hat{\mathbb{E}}[\xi]=\sup_{P\in\mathcal{P}}E_{P}[\xi], \quad \forall \xi\in L_{G}^{1}(\Omega).
	\end{align} 
\end{theorem}

\begin{remark}
	Under each $ P\in\mathcal{P} $, the $ G $-Brownian motion $ (B_{t})_{t\geq0} $ is a martingale.
\end{remark}

For this $ \mathcal{P} $, we define capacity
\[\mathrm{c}(A):=\sup_{P\in\mathcal{P}}P(A), \ A\in\mathcal{F}.\]

\begin{definition}
	A set $ A $ is polar if $ \mathrm{c}(A)=0 $ and a property holds “quasi-surely” (q.s.) if it holds outside a polar set.
\end{definition}

\begin{definition}
	Let $ M_{G}^{0}(0,T) $ be the collection of processes of the following form: for a given partition $ \{t_{0},\cdots,t_{N}=\pi_{T}\} $ of $ [0,T] $,
	\[\eta_{t}(\omega)=\sum_{i=0}^{N-1}\xi_{i}(\omega)I_{[t_{i},t_{i+1})}(t),\]
	where $ \xi_{i}\in Lip(\Omega_{t_{i}}), i=0,1,\ldots,N-1 $. For each $ p\geq1 $, we denote by $ M_{G}^{p}(0,T) $ the completion of $ M_{G}^{0}(0,T) $ under the norm $ \|\eta\|_{M_{G}^{p}}:=(\hat{\mathbb{E}}[\int_{0}^{T}|\eta_{t}|^{p}\mathrm{d}t])^{\frac{1}{p}} $.
\end{definition}

According to \cite{li2011stopping,peng2007G,peng2008multi,peng2019nonlinear}, the $ G $-It\^{o} integral $ \int_{0}^{t}\eta_{s}\mathrm{d}B_{s}, \ \int_{0}^{t}\xi_{s}\mathrm{d}\langle B\rangle_{s} $ are will-defined for $ \eta\in M_{G}^{2}(0,T) $ and $ \xi\in M_{G}^{1}(0,T) $.

Now we consider the following SDE on the $ G $-expectation space $(\Omega ,L_{G}^{1}(\Omega),\hat{\mathbb{E}})  $ and $ B $ is a one-dimensional $ G $-Brownian motion: for each given $ 0\leq T<\infty $ and $ x_{0}\in\mathbb{R} $,
\begin{align}\label{GSDE}
	\begin{cases}
		\mathrm{d}X_{t}=b(t,X_{t})\mathrm{d}t+h(t,X_{t})\mathrm{d}\langle B\rangle_{s}+\sigma(t,X_{t})\mathrm{d}B_{t},  \quad t\in[0,T],\\
		X_{0}=x_{0},
	\end{cases}		
\end{align}
where $ b,h,\sigma: \mathbb{R}^{2}\rightarrow\mathbb{R} $ are given deterministic functions satisfying the following conditions:

\begin{description}
	\item[(H1)] $ b(\cdot,x),h(\cdot,x),\sigma(\cdot,x)\in M_{G}^{2}(0,T) $ for each $ x\in\mathbb{R} $;
	\item[(H2)] The functions $ b,h,\sigma $ are Lipschitz in $ x $ with Lipschitz constant $ L $.
\end{description}

\begin{theorem}[\cite{peng2019nonlinear}]\label{thmsde}
	If (H1) and (H2) hold, then the stochastic differential equation (\ref{GSDE}) has a unique solution $ X\in M_{G}^{2}(0,T) $.
\end{theorem}

\section{$ G $-Bessel processes}\label{s3}

In this section, we introduce the notion of $ G $-Bessel processes. For each integer $ d\geq2 $, we consider a $ d $-dimensional $ G $-Brownian motion $ B=(B^{(1)},\cdots,B^{(d)}) $. For each $ x\in\mathbb{R}^{d} $, let $ B^{x}=B+x $ be a $ d $-dimensional $ G $-Brownian motion starting at $ x $ and let $ B^{x}=(\bar{B}^{(1)},\cdots,\bar{B}^{(d)}) $. Note that $ \mathrm{d}B^{(i)}_{t}=\mathrm{d}\bar{B}^{(i)}_{t},\ \mathrm{d}\langle B^{(i)},B^{(j)}\rangle_{t}=\mathrm{d}\langle \bar{B}^{(i)},\bar{B}^{(j)}\rangle_{t} $ for $ 1\leq i,j\leq d$. For notational simplicity, we write $ \mathrm{d}B^{(i)}_{t},\ \mathrm{d}\langle B^{(i)},B^{(j)}\rangle_{t} $ instead of $ \mathrm{d}\bar{B}^{(i)}_{t},\ \mathrm{d}\langle \bar{B}^{(i)},\bar{B}^{(j)}\rangle_{t} $. Analogous to the classical case, we define the non-negative process by
\[R_{t}:=|B_{t}^{x}|,\quad 0\leq t<\infty.\]
Roughly speaking, by applying the It\^{o} formula to $ R_{t} $, we obtain that
\begin{align}\label{eq3.1}
	\begin{cases}
		\mathrm{d}R_{t}=\sum\limits_{i=1}^{d}\frac{\bar{B}_{t}^{(i)}}{R_{t}}\mathrm{d}B_{t}^{(i)}+\sum\limits_{i=1}^{d}\frac{1}{2R_{t}}\mathrm{d}\langle B^{(i)} \rangle_{t}-\sum\limits_{i=1}^{d}\frac{(\bar{B}_{t}^{(i)})^{2}}{2R_{t}^{3}}\mathrm{d}\langle B^{(i)} \rangle_{t}-\sum\limits_{1\leq i<j\leq d}\frac{\bar{B}_{t}^{(i)}\bar{B}_{t}^{(j)}}{R_{t}^{3}}\mathrm{d}\langle B^{(i)},B^{(j)}\rangle_{t},\\
		R_{0}=|x|.
	\end{cases}
\end{align}

According to Theorem \ref{Levy}, the following assumption is required for $ \beta_{t}=\sum_{i=1}^{d}\int_{0}^{t}\frac{\bar{B}_{s}^{(i)}}{R_{s}}\mathrm{d}B_{s}^{(i)},\ t\geq0 $, to be a one-dimensional $ G' $-Brownian motion where $ G'(a)=\frac{1}{2}(\bar{\sigma}^{2}a^{+}-\underline{\sigma}^{2}a^{-}),\ \bar{\sigma}^{2}=\hat{\mathbb{E}}[\langle \beta \rangle_{1}],\ \underline{\sigma}^{2}=-\hat{\mathbb{E}}[-\langle \beta \rangle_{1}], \ \forall a\in\mathbb{R}  $.  
\begin{description}
	\item[(H)] For $ t\geq0 $ and $ 1\leq i,j\leq d,\ i\neq j $,
	\[\langle B^{(1)} \rangle_{t}=\langle B^{(i)} \rangle_{t},\quad \langle B^{(i)},B^{(j)}\rangle_{t}=0. \]
\end{description}
Due to assumption (H), we obtain that $ \langle \beta \rangle_{t}=\langle B^{(1)} \rangle_{t} $. Briefly speaking, if (H) holds, then $ \beta $ can be considered as a one-dimensional $ G' $-Brownian motion. Furthermore, equation (\ref{eq3.1}) reduces to 

\begin{align}\label{formally}
	\begin{cases}
		\mathrm{d}R_{t}=\mathrm{d}\beta_{t}+\frac{d-1}{2R_{t}}\mathrm{d}\langle \beta \rangle_{t},\\
		R_{0}=|x|.
	\end{cases}
\end{align}

\begin{remark}
	It is important to note that assumption (H) is not a necessary condition for $ \beta $ to be a one-dimensional $ G' $-Brownian motion. For example, let $ d=2 $ and let $ G(A)=G'(a_{11})+G'(a_{22}) $ where $ G'(x)=\frac{1}{2}(\bar{\sigma}^{2}x^{+}-\underline{\sigma}^{2}x^{-}) $ for $ x\in \mathbb{R} ,\ 0<\underline{\sigma}\leq\bar{\sigma}<\infty $.. Then $ \langle B^{(1)},B^{(2)}\rangle_{t}=0,\ \langle B^{(1)}\rangle_{t}\neq\langle B^{(2)}\rangle_{t} $ for $ t\geq0 $. Now $\langle \beta\rangle_{t}=\int_{0}^{t}\frac{(\bar{B}^{(1)}_{s})^{2}}{R_{s}^{2}}\mathrm{d}\langle B^{(1)}\rangle_{s}+\int_{0}^{t}\frac{(\bar{B}^{(2)}_{s})^{2}}{R_{s}^{2}}\mathrm{d}\langle B^{(2)}\rangle_{s}  $. Denote $ \frac{(\bar{B}^{(1)}_{s})^{2}}{R_{s}^{2}},\  \frac{(\bar{B}^{(2)}_{s})^{2}}{R_{s}^{2}} $ by $ \tilde{a}_{s},\  \tilde{b}_{s} $. From Theorem 4.16 of \cite{hu2016quasi}, we obtain that $ \tilde{a}, \tilde{b}\in M_{G}^{2}(0,T) $. Then let $ \tilde{a}_{s}^{n}=\sum_{k=0}^{n-1}\xi_{k}I_{[t_{k},t_{k+1} )}(s), \ \xi_{k}\in[0,1], \ 0=t_{0}<t_{1}<\cdots<t_{n}=t $ and $ \tilde{b}_{s}^{n}=1-\tilde{a}_{s}^{n} $ be simple processes such that $ \tilde{a}^{n}, \tilde{b}^{n} $ converge to $\tilde{a}, \tilde{b}  $ in $ M_{G}^{2}(0,T) $ as $ n\rightarrow\infty $. By Theorem \ref{quad B}, we get
	\begin{align*}
		{}&\hat{\mathbb{E}}_{t_{n-1}}\left[\int_{t_{n-1}}^{t_{n}}\tilde{a}_{r}^{n}\mathrm{d}\langle B^{(1)}\rangle_{r}+\int_{t_{n-1}}^{t_{n}}\tilde{b}_{r}^{n}\mathrm{d}\langle B^{(2)}\rangle_{r}\right]\\
		={}&\hat{\mathbb{E}}_{t_{n-1}}[\xi_{t_{n-1}}(\langle B^{(1)}\rangle_{t_{n}}-\langle B^{(1)}\rangle_{t_{n-1}})+(1-\xi_{t_{n-1}})(\langle B^{(2)}\rangle_{t_{n}}-\langle B^{(2)}\rangle_{t_{n-1}})]\\
		={}&\hat{\mathbb{E}}_{t_{n-1}}[x(\langle B^{(1)}\rangle_{t_{n}}-\langle B^{(1)}\rangle_{t_{n-1}})+(1-x)(\langle B^{(2)}\rangle_{t_{n}}-\langle B^{(2)}\rangle_{t_{n-1}})]_{x=\xi_{t_{n-1}}}\\
		={}&\bar{\sigma}^{2}(t_{n}-t_{n-1}).
	\end{align*} 
	Then we can continuously extend the above equality to the case $ \tilde{a}, \tilde{b} $ and get
	\[\hat{\mathbb{E}}_{s}\left[\int_{s}^{t}\tilde{a}_{r}\mathrm{d}\langle B^{(1)}\rangle_{r}+\int_{s}^{t}\tilde{b}_{r}\mathrm{d}\langle B^{(2)}\rangle_{r}\right]=\bar{\sigma}^{2}(t-s).\] 
	Therefore, one can check that $ \{\beta_{t}^{2}-\bar{\sigma}^{2}t\}_{t\geq0} $ is a martingale. Similarly, $ \{-\beta_{t}^{2}+\underline{\sigma}^{2}t\}_{t\geq0} $ is also a martingale. By Theorem \ref{Levy}, $ \beta $ is a one-dimensional $ G' $-Brownian motion. In this case , since $ \langle B^{(1)}\rangle\neq\langle B^{(2)}\rangle $, equation (\ref{eq3.1}) takes the following form,
    \begin{align*}
    	\begin{cases}
    		\mathrm{d}R_{t}=\mathrm{d}\beta_{t}+\frac{\mathrm{d}\langle B^{(1)} \rangle_{t}+\mathrm{d}\langle B^{(2)} \rangle_{t}}{2R_{t}}-\frac{1}{2R_{t}}\mathrm{d}\langle \beta \rangle_{t},\\
    		R_{0}=|x|.
    	\end{cases}
    \end{align*}
    We cannot obtain the equation (\ref{formally}) satisfied by $ R $. In order to obtain equation (\ref{formally}), we need condition $ \langle B^{(1)}\rangle=\langle B^{(2)}\rangle =\langle \beta\rangle $. Then we obtain assumption (H). 
\end{remark}

According to Theorem \ref{quad B}, the function $ G $ must satisfy the following condition for (H) to hold:

\begin{description}
	\item[(A)]  
	For each $ A=(a_{ij})_{1\leq i\leq j\leq d}\in \mathbb{S}_{d} $, 
	\[ G(A)=G'(\sum_{i=1}^{d}a_{ii}),  \] where $G'(a)=\frac{1}{2}(\bar{\sigma}^{2}a^{+}-\underline{\sigma}^{2}a^{-})$ for $ a\in \mathbb{R} ,\ 0<\underline{\sigma}\leq\bar{\sigma}<\infty $.
	
\end{description}

\begin{lemma}\label{lem4.1}
	Let $ B_{t}=(B_{t}^{(1)},\cdots,B_{t}^{(d)}),\ t\geq0 $, be the $ d $-dimensional $ G $-Brownian motion. Then the following conditions are equivalent: 
	\begin{description}
		\item[(1)] The function $ G $ satisfies the condition (A);
		\item[(2)] The quadratic variation process of $ B $ satisfies (H).
	\end{description}
	
\end{lemma}

\begin{proof} 
	$ (1)\Rightarrow(2)  $ It is clear that  
	\[ G(A)=\frac{1}{2}\sup_{\gamma\in \Gamma}\mathrm{tr}[\gamma A], \quad \text{for } A\in\mathbb{S}_{d},\]
	where $ \Gamma=\{\mathrm{diag}[\nu,\cdots,\nu]: \nu\in[\underline{\sigma}^{2},\bar{\sigma}^{2}]\} $.
	Now for each fixed $ i $, let $ \varphi(A)=\left| a_{11}-a_{ii}\right|  $. By Theorem \ref{quad B}, we yield
	\[\hat{\mathbb{E}}[\varphi(\left\langle B\right\rangle_{t} )]=\sup_{\gamma\in \Gamma}\varphi(t\gamma)=0.\]
	Since $ \varphi\geq0 $, it follows that $ \left| \langle B^{(1)} \rangle_{t}-\langle B^{(i)} \rangle_{t}\right|=0 $, q.s. Next for each fixed $ i,j $ and $ i\neq j $, set $ \phi(A)=|a_{ij}| $. Also by Theorem \ref{quad B}, we obtain
	\[ \hat{\mathbb{E}}[\phi(\left\langle B\right\rangle_{t})]=\sup_{\gamma\in \Gamma}\phi(t\gamma)=0.  \]
	Due to the fact that $ \phi\geq0 $, it follows that $ \left| \langle B^{(i)},B^{(j)}\rangle_{t} \right|=0 $ q.s.
	
	$ (2)\Rightarrow(1) $ Let us define $ \varphi $ and $ \phi $ as above. By Theorem \ref{quad B}, $ \sup_{\gamma\in \Gamma}\varphi(t\gamma)=0 $. Notice that $ \varphi\geq0 $. Then we obtain $\varphi(t\gamma)=0  $ for each $ \gamma\in \Gamma $. Consequently, the diagonal elements of the matrix $ \gamma\in\Gamma $ must be identical. Similarly, $ \phi(t\gamma)=0 $ holds for every $ \gamma\in\Gamma $. Therefore, all off-diagonal elements of the matrix $ \gamma\in\Gamma $ must be zero. Obviously $\Gamma=\{\mathrm{diag}[\nu,\cdots,\nu]: \nu\in[\underline{\sigma}^{2},\bar{\sigma}^{2}]\}   $, then  \[G(A)=\frac{1}{2}\sup_{\gamma\in \Gamma}\mathrm{tr}[\gamma A]=\sup_{\nu\in[\underline{\sigma}^{2},\bar{\sigma}^{2}]}\frac{\nu}{2}\mathrm{tr}[A]. \] 
	This completes the proof.
\end{proof}

\begin{lemma}[rotation invariance]\label{lem3.2}
	Let $ G $ be a function satisfying the condition (A) and let $ B^{x}=B+x $ be the  $ d $-dimensional $ G $-Brownian motion starting at $ x $. Then for any $ d\times d $ orthogonal matrix $ Q $, the process $ X_{t}=QB_{t}^{x}, t\geq0 $, is a $ d $-dimensional $ G $-Brownian motion starting at $ Qx $.
\end{lemma}

\begin{proof}
	By Proposition 1.5 in Chapter \Rmnum{3} of \cite{peng2019nonlinear}, we know that $ (X_{t})_{t\geq0} $ is a $ d $-dimensional $ G_{Q} $-Brownian motion starting at $ Qx $, where 
	\begin{align*}
		G_{Q}(A)={}&\frac{1}{2}\hat{\mathbb{E}}[\langle AQB_{1},QB_{1} \rangle]\\
		={}&\frac{1}{2}\hat{\mathbb{E}}[\langle Q^{T}AQB_{1},B_{1} \rangle]\\
		={}&G(Q^{T}AQ)\\
		={}&G(A),\quad A\in\mathbb{S}_{d}.
	\end{align*}
    This completes the proof.
\end{proof}

\begin{remark}
	Let $ B^{y} $ be the  $ d $-dimensional $ G $-Brownian motion starting at $ y $ and let $ B^{x} $ be the  $ d $-dimensional $ G $-Brownian motion starting at $ x $. If $ |y|=|x| $, then there exists an orthogonal matrix $ Q $ such that $ y=Qx $. In the above lemma, we obtain that $ QB^{x} $ is a $ G $-Brownian motion starting at $ y $. Since $ |QB^{x}|=|B^{x}| $, it follows that 
	\[ |B^{y}|\overset{d}{=}|B^{x}| . \]
	Therefore, the law of $ |B^{x}| $ depends on $ |x| $. 
\end{remark}

We now introduce the definition of $ {G} $-Bessel processes.

\begin{definition}[$ G $-Bessel processes]\label{def3.4}
	Let $ G $ be a function satisfying the condition (A). For $ x\in\mathbb{R}^{d} $,  let $ B $ be a $ d $-dimensional $ G $-Brownian motion on the $ G $-expectation space $(\Omega ,L_{G}^{1}(\Omega),\hat{\mathbb{E}})  $ and let $ B^{x}=(\bar{B}^{(1)},\cdots,\bar{B}^{(d)})=B+x $ be a $ d $-dimensional $ G $-Brownian motion starting at $ x $.
	\[R_{t}:=|B_{t}^{x}|=\sqrt{(\bar{B}_{t}^{(1)})^{2}+\cdots+(\bar{B}_{t}^{(d)})^{2}}, \quad 0\leq t<\infty , \]
	is called a $ G $-Bessel process with dimensional $ d $ starting at $ r=|x| $ and denote by $ G $-$\mathrm{BES} _{\mathrm{r}}^{\mathrm{d}} $.
\end{definition}	

\begin{corollary}
	The $ G $-Bessel process has the Brownian scaling property, i.e., if $ (R_{t})_{t\geq0} $ is a $ G $-$\mathrm{BES} _{\mathrm{r}}^{\mathrm{d}} $, then the $ (\lambda^{-1}R_{\lambda^{2} t})_{t\geq0} $ is a $ G $-$\mathrm{BES} _{\mathrm{\lambda^{-1}r}}^{\mathrm{d}} $ for any $ \lambda>0 $.
\end{corollary}

\begin{proof} 
	Changing the variable of integration, we get
	\[\lambda^{-1}R_{\lambda^{2} t}=\sqrt{(\lambda^{-1}\bar{B}_{\lambda^{2}t}^{(1)})^{2}+\cdots+(\lambda^{-1}\bar{B}_{\lambda^{2}t}^{(d)})^{2}}, \quad 0\leq t<\infty . \]
	By Remark 1.4 in Chapter \Rmnum{3} of \cite{peng2019nonlinear}, $ (\lambda^{-1}B_{\lambda^{2}t}^{x})_{t\geq0} $ is a $ d $-dimensional $ G $-Brownian motion starting at $ \lambda^{-1}x $, then we deduce the result. 
\end{proof}

\section{Some properties of $ G $-Bessel processes}\label{s4}

In this section,we present the main result of $ G $-Bessel processes. Let $ B $ and $ B^{x} $ be defined as in Definition \ref{def3.4}. 

\begin{theorem}\label{thm4.1}
	Let $ d\geq [\frac{\bar{\sigma}^{2}}{\underline{\sigma}^{2}}]+1$ and $ T\geq0 $. Let $( R_{t})_{t\in [0,T]} $ be a $ G $-$\mathrm{BES} _{\mathrm{r}}^{\mathrm{d}} $ for given $ r\geq0 $ . Then $ \frac{1}{R} \in M_{G}^{1}(0,T) $ and $  R $ satisfies the following equation:
	\begin{align}\label{R}
		\begin{cases}
			\mathrm{d}R_{t}=\frac{d-1}{2R_{t}}\mathrm{d}\langle \beta \rangle_{t}+\mathrm{d}\beta_{t},\\
			R_{0}=r,
		\end{cases}
	\end{align}
	where 
	\begin{align}\label{B}
		\beta_{t}:=\sum_{i=1}^{d}\int_{0}^{t}\frac{\bar{B}_{s}^{(i)}}{R_{s}}\mathrm{d}B_{s}^{(i)}, \quad t\in[0,T],
	\end{align}
    is a one-dimensional $ G' $-Brownian motion. Moreover, if $ d\geq([\frac{\bar{\sigma}^{2}}{\underline{\sigma}^{2}}]+1)\vee3 $ and $ r>0 $, then $ R\in\tilde{M}_{G}^{2}(0,T) $ is the unique solution of equation (\ref{R}) and never reaches the origin q.s., that is,
	\[\mathrm{c}(\{R_{t}=0,\quad t\in [0,\infty)\} )=0. \]
\end{theorem}

\begin{remark}
	Let us clarify the space to which the solution belongs. Since $ B $ and $ B^{x}=B+x $ are two $ G $-Brownian motions with different initial values, we point out that the solution of equation (\ref{R}) is considered in the space $ M_{G}^{2}(0,T) $ which is generated by $ d $-dimensional $ G $-Brownian motion $ B $. 
	
	Since $ \beta $ is a one-dimensional $ G' $-Brownian motion, we can similarly define the following spaces:
	\begin{itemize}
		\item $ \tilde{Lip}(\Omega_{t}):=\{\varphi(\beta_{t_{1}},\cdots,\beta_{t_{n}}): n\in\mathbb{N}, 0\leq t_{1}\leq\cdots\leq t_{n}\leq t,\varphi\in C_{b,Lip}(\mathbb{R}^{1\times n})\} $;
		\item $ \tilde{Lip}(\Omega):=\cup_{n=1}^{\infty}Lip(\Omega_{n}) $;
		\item $ \tilde{M}_{G}^{0}(0,T):=\{\eta_{t}(\omega)=\sum_{i=0}^{N-1}\xi_{i}(\omega)I_{[t_{i},t_{i+1})}(t): n\in\mathbb{N},0=t_{0}< t_{1}<\cdots< t_{n}= T, \xi_{i}\in\tilde{Lip}(\Omega_{t_{i}})  \} $.
	\end{itemize}
    Define $ \|X\|_{\tilde{L}_{G}^{p}}=(\hat{\mathbb{E}}[|X|^{p}])^{\frac{1}{p}} $ and $\|\eta\|_{\tilde{M}_{G}^{p}}=(\hat{\mathbb{E}}[\int_{0}^{T}|\eta_{t}|^{p}\mathrm{d}t])^{\frac{1}{p}}  $ for $ p\geq1,\  X\in\tilde{Lip}(\Omega),\ \eta\in\tilde{M}_{G}^{0}(0,T) $. Then we can define $  \tilde{L}_{G}^{p}(\Omega),\ \tilde{M}_{G}^{p}(0,T) $ analogously to $ L_{G}^{p}(\Omega),\ M_{G}^{p}(0,T) $. Since $ \beta $ is generated by $ B^{x} $, we have $ \tilde{L}_{G}^{p}(\Omega)\subset L_{G}^{p}(\Omega),\ \tilde{M}_{G}^{p}(0,T)\subset M_{G}^{p}(0,T)  $. 
	
	If $ d\geq [\frac{\bar{\sigma}^{2}}{\underline{\sigma}^{2}}]+1 $, we can prove that $ R=|B^{x}| $ is a solution which belongs to $ M_{G}^{2}(0,T) $. If $ d\geq([\frac{\bar{\sigma}^{2}}{\underline{\sigma}^{2}}]+1)\vee3 $ and $ r>0 $, then from Lemma \ref{lem4.6}, we know $ R $ is a unique solution of (\ref{R}) which belongs to $ M_{G}^{2}(0,T) $. Moreover, $R\in \tilde{M}_{G}^{2}(0,T) $.
\end{remark}

To prove the above theorem, we need the following lemmas. 

\begin{lemma}
	Let $ G $ be a function satisfying the condition (A) and let $ T\geq0 $. Then $\beta_{t}=\sum_{i=1}^{d} \int_{0}^{t}\frac{\bar{B}_{s}^{(i)}}{R_{s}}\mathrm{d}B_{s}^{(i)} $, $t\in[0,T],$ is a one-dimensional $ G' $-Brownian motion and for $ 0\leq s\leq t\leq T $, $\beta_{t}-\beta_{s}  $ is independent from $ L_{G}^{1}(\Omega_{s}) $.
\end{lemma}

\begin{proof} 
	Let us define $ \beta_{t}^{(i)}=\int_{0}^{t}\frac{\bar{B}_{s}^{(i)}}{R_{s}}\mathrm{d}B_{s}^{(i)}$ for $ 1\leq i\leq d $. Set $ \varphi(x)=\frac{x_{i}}{|x|}I_{\{x_{i}\neq0\}} $ where $ x=(x_{1},\cdots,x_{d})\in\mathbb{R}^{d} $. Then $ \varphi $ is a bounded Borel measurable function and $ \varphi(B_{s}^{x})=\frac{ \bar{B}_{s}^{(i)}}{R_{s}}I_{\{\bar{B}_{s}^{(i)}\neq0\}} $. From Theorem 4.16 of \cite{hu2016quasi}, we obtain $ ( \varphi(B_{s}^{x}))_{s\in[0,T]}\in M_{G}^{2}(0,T) $. Consequently, $(\frac{ \bar{B}_{s}^{(i)}}{R_{s}})_{s\in[0,T]} \in M_{G}^{2}(0,T) $. By applying BDG's inequality (Lemma 1.12 in Chapter \Rmnum{8} of \cite{peng2019nonlinear}), we obtain
	\begin{align*}
		\hat{\mathbb{E}}[|\beta_{t}^{(i)}|^{4}]\leq{}&C\bar{\sigma}^{2}\hat{\mathbb{E}}\left[\left( \int_{0}^{t}(\frac{\bar{B}_{s}^{(i)}}{R_{s}})^{2}\mathrm{d}s\right) ^{2} \right]\\
		\leq{}& C\bar{\sigma}^{2}t^{2}.
	\end{align*}
	From Theorem 54 of \cite{denis2011function}, it follows that $ \beta_{t}^{(i)} \in L_{G}^{3}(\Omega_{t}) $. Hence, $ \beta_{t}\in L_{G}^{3}(\Omega_{t}) $. Since $\hat{\mathbb{E}}[\beta_{t}]=\hat{\mathbb{E}}[-\beta_{t}]=0$, $ \beta $ is a symmetric martingale. Using H\"{o}lder's inequality and BDG's inequality, we yield 
	\begin{align*}
		\hat{\mathbb{E}}[|\beta_{t+\epsilon}-\beta_{t}|^{3}]={}&\hat{\mathbb{E}}\left[\left| \sum_{i=1}^{d} \int_{t}^{t+\epsilon}\frac{\bar{B}_{s}^{(i)}}{R_{s}}\mathrm{d}B_{s}^{(i)} \right|^{3} \right]\\
		\leq{}&d^{2}\sum_{i=1}^{d}\hat{\mathbb{E}}\left[\left|  \int_{t}^{t+\epsilon}\frac{\bar{B}_{s}^{(i)}}{R_{s}}\mathrm{d}B_{s}^{(i)} \right|^{3} \right] \\
		\leq{}&Cd^{2}\bar{\sigma}^{2}\sum_{i=1}^{d}\hat{\mathbb{E}}\left[\left( \int_{t}^{t+\epsilon}(\frac{\bar{B}_{s}^{(i)}}{R_{s}})^{2}\mathrm{d}s\right) ^{\frac{3}{2}} \right]\\
		\leq{}&Cd^{3}\bar{\sigma}^{2}\epsilon^{\frac{3}{2}}.
	\end{align*}
	From lemma \ref{lem4.1}, we have
	\[ \langle \beta\rangle_{t}=\sum_{i=1}^{d}\int_{0}^{t}(\frac{\bar{B}_{s}^{(i)}}{R_{s}})^{2}\mathrm{d}\langle B^{(i)}\rangle_{s}=\langle B^{(1)}\rangle_{t}, \quad t\geq0. \]
	Notice that $ \beta $ is a symmetric martingale. Then we obtain $ \hat{\mathbb{E}}_{s}[\beta_{s}(\beta_{t}-\beta_{s})]=\hat{\mathbb{E}}_{s}[-\beta_{s}(\beta_{t}-\beta_{s})]=0 $. Using Proposition 3.4 of \cite{hu2019levy}, it follows that
	\begin{align*}
		\hat{\mathbb{E}}_{s}[\beta_{t}^{2}-\bar{\sigma}^{2}t]={}&\hat{\mathbb{E}}_{s}[(\beta_{t}-\beta_{s}+\beta_{s})^{2}]-\bar{\sigma}^{2}t\\
		={}&\hat{\mathbb{E}}_{s}[(\beta_{t}-\beta_{s})^{2}+2\beta_{s}(\beta_{t}-\beta_{s})+\beta_{s}^{2}]-\bar{\sigma}^{2}t\\
		={}&\hat{\mathbb{E}}_{s}[(\beta_{t}-\beta_{s})^{2}]+\beta_{s}^{2}-\bar{\sigma}^{2}t\\
		={}&\hat{\mathbb{E}}_{s}[\langle B^{(1)}\rangle_{t}-\langle B^{(1)}\rangle_{s}]+\beta_{s}^{2}-\bar{\sigma}^{2}t\\
		={}&\beta_{s}^{2}-\bar{\sigma}^{2}s.
	\end{align*}
	Therefore, $ \{\beta_{t}^{2}-\bar{\sigma}^{2}t\}_{t\geq0} $ is a martingale. Similarly, $ \{-\beta_{t}^{2}+\underline{\sigma}^{2}t\}_{t\geq0} $ is also a martingale. By Theorem \ref{Levy}, $ \beta $ is a one-dimensional $ G' $-Brownian motion. Now  for each $ \varphi\in C_{b,Lip}(\mathbb{R}) $, from the proof of Theorem 4.1 of \cite{hu2019levy}, we know
	\[\hat{\mathbb{E}}_{s}[\varphi(\beta_{t}-\beta_{s})]=\hat{\mathbb{E}}[\varphi(\beta_{t}-\beta_{s})].\]
	Thus for each $ \xi\in (L_{G}^{1}(\Omega_{s}))^{d} $ and $ \phi\in C_{b,Lip}(\mathbb{R}^{d+1}) $, by applying Proposition 3.5 of \cite{hu2019levy}, we have
	\begin{align*}
		\hat{\mathbb{E}}[\phi(\xi,\beta_{t}-\beta_{s})]={}&\hat{\mathbb{E}}[\hat{\mathbb{E}}_{s}[\phi(\xi,\beta_{t}-\beta_{s})]]\\
		={}&\hat{\mathbb{E}}[\hat{\mathbb{E}}_{s}[\phi(c,\beta_{t}-\beta_{s})]_{c=\xi}]\\
		={}&\hat{\mathbb{E}}[\hat{\mathbb{E}}[\phi(c,\beta_{t}-\beta_{s})]_{c=\xi}].
	\end{align*}
	The proof is complete.
\end{proof}

Now we give an estimate of the solution to $ G $-heat equation, which will be needed in what follows.

\begin{lemma}\label{lem4.4}
	Let $ G $ be a function satisfying the condition (A). Suppose $ u_{n} $ is the solution of the following PDE defined on $ [0,\infty)\times\mathbb{R}^{d} $,
	\begin{equation}\label{equn}
		\begin{cases}
			\partial_{t}u_{n}(t,x)-G(D^{2}_{x}u_{n}(t,x))=0,\\
			u_{n}(0,x)=\exp\left(-\frac{n| x-a|^{2}}{2\bar{\sigma}^{2}}\right),
		\end{cases}
	\end{equation} 
	where $ n>0 $ and $ a\in\mathbb{R}^{d} $. Then for each fixed $ c\in[0,d] $, we obtain
	\[ u_{n}(t,x)\leq(1+nt)^{-c\rho}, \]
	where $ \rho=\frac{\underline{\sigma}^{2}}{2\bar{\sigma}^{2}} $.
\end{lemma}

\begin{proof} 
	Consider $ v_{n}(t,x)=(1+nt)^{-c\rho}\exp\left(-\frac{n| x-a|^{2}}{2(1+nt)\bar{\sigma}^{2}}\right) $. It is simple to obtain that
	\begin{align*}
		&\partial_{t}v_{n}(t,x)=-\frac{nc\rho}{1+nt}v_{n}(t,x)+\frac{n^{2}| x-a|^{2}}{2\bar{\sigma}^{2}(1+nt)^{2}}v_{n}(t,x),\\
		&\partial_{x_{i}}v_{n}(t,x)=-\frac{n(x_{i}-a_{i})}{\bar{\sigma}^{2}(1+nt)}v_{n}(t,x),\quad i=1,\cdots,d,\\
		&\partial^{2}_{x_{i}x_{i}}v_{n}(t,x)=-\frac{n}{\bar{\sigma}^{2}(1+nt)}v_{n}(t,x)+\frac{n^{2}| x_{i}-a_{i}|^{2}}{(\bar{\sigma}^{2}(1+nt))^{2}}v_{n}(t,x).
	\end{align*}
	Therefore, 
	\begin{align*}
		\partial_{t}v_{n}(t,x)-G(D^{2}_{x}v_{n}(t,x))={}&\partial_{t}v_{n}(t,x)-G'\left(\sum_{i=1}^{d} \partial_{x_{i}x_{i}}^{2}v_{n}(t,x)\right)\\
		={}&-\frac{nc\rho}{1+nt}v_{n}(t,x)+\frac{n^{2}| x-a|^{2}}{2\bar{\sigma}^{2}(1+nt)^{2}}v_{n}(t,x)\\
		&-G'\left(\sum_{i=1}^{d}\left(-\frac{n}{\bar{\sigma}^{2}(1+nt)}v_{n}(t,x)+\frac{n^{2}| x_{i}-a_{i}|^{2}}{(\bar{\sigma}^{2}(1+nt))^{2}}v_{n}(t,x)\right)\right)\\
		\geq{}&-\frac{nc\rho}{1+nt}v_{n}(t,x)+\frac{n^{2}| x-a|^{2}}{2\bar{\sigma}^{2}(1+nt)^{2}}v_{n}(t,x)\\
		&-G'\left(-\frac{dn}{\bar{\sigma}^{2}(1+nt)}v_{n}(t,x)\right)-G'\left(\frac{n^{2}| x-a|^{2}}{(\bar{\sigma}^{2}(1+nt))^{2}}v_{n}(t,x)\right)\\
		={}&\frac{n\rho}{1+nt}v_{n}(t,x)(d-c)\\
		\geq{}& 0,
	\end{align*}
	which implies that $ v_{n} $ is a bounded supersolution of (\ref{equn}). According to the comparison theorem, we have
	\[ u_{n}(t,x)\leq v_{n}(t,x)\leq (1+nt)^{-c\rho}. \]
	The proof is complete.
\end{proof}

The following Lemma plays a crucial role in our approach.

\begin{lemma}\label{im}
	Let $ d\geq [\frac{\bar{\sigma}^{2}}{\underline{\sigma}^{2}}]+1 $ and let $ G $ be a function satisfying the condition (A). Then for any	  $ T,\ \epsilon>0$, and $a\in\mathbb{R}^{d} $, we have 
	\[\lim_{\epsilon\downarrow 0}\hat{\mathbb{E}}\left[\int_{0}^{T}\frac{1}{\epsilon}I_{\{| B_{t}-a|<\epsilon\}}\mathrm{d}t \right]=0. \]
\end{lemma}

\begin{proof} 
	Set $ n>0 $, note that 
	\[\hat{\mathbb{E}}\left[I_{\{| B_{t}-a|<\epsilon\}} \right]\leq\exp(\frac{n\epsilon^{2}}{2\bar{\sigma}^{2}})\hat{\mathbb{E}}\left[\exp(-\frac{n| B_{t}-a|^{2}}{2\bar{\sigma}^{2}}) \right].  \]
	Since 
	\[  \mathrm{c}(\{| B_{t}-a|<\epsilon\})=\hat{\mathbb{E}}\left[I_{\{| B_{t}-a|<\epsilon\}} \right]. \]
	According to Lemma \ref{lem4.4}, for each fixed $ c\in[0,d] $, we get
	\[\mathrm{c}(\{| B_{t}-a|<\epsilon\})\leq\exp(\frac{n\epsilon^{2}}{2\bar{\sigma}^{2}})(1+nt)^{-c\rho},  \]
	where $ \rho=\frac{\underline{\sigma}^{2}}{2\bar{\sigma}^{2}} $. Choosing $ n=\frac{1}{\epsilon^{2}} $, it follows that 
	\[\mathrm{c}(\{| B_{t}-a|<\epsilon\})\leq\exp(\frac{1}{2\bar{\sigma}^{2}})\frac{\epsilon^{2c\rho}}{t^{c\rho}}.\]
	Note that $ d\rho>\frac{1}{2} $. Then we can choose $ c $ such that $ \alpha=c\rho\in(\frac{1}{2},1) $. It follows that
	\begin{align*}
		\hat{\mathbb{E}}\left[\int_{0}^{T}\frac{1}{\epsilon}I_{\{| B_{t}-a|<\epsilon\}}\mathrm{d}t \right]\leq{}&\int_{0}^{T}\frac{1}{\epsilon}\mathrm{c}(\{| B_{t}-a|<\epsilon\})\mathrm{d}t\\
		\leq{}&\frac{1}{1-\alpha}\exp(\frac{1}{2\bar{\sigma}^{2}})\epsilon^{2\alpha-1}T^{1-\alpha}\rightarrow0 \quad \text{as } \epsilon\downarrow 0.
	\end{align*}
	The proof is complete.
\end{proof}

\begin{remark}
	Let $ B^{x}=B+x $ be a $ d $-dimensional $ G $-Brownian motion starting at $ x $ where $ x\in\mathbb{R}^{d} $. If we choose $ a=-x $ in above lemma, then we obtain $ \lim_{\epsilon\downarrow 0}\hat{\mathbb{E}}\left[\int_{0}^{T}\frac{1}{\epsilon}I_{\{| B_{t}^{x}|<\epsilon\}}\mathrm{d}t \right]=0 $.
\end{remark}

We now give the existence and uniqueness of the solution to the stochastic differential equation with non-Lipschitz coefficient driven by the $ G' $-Brownian motion.

\begin{lemma}\label{lem4.6}
	Let $ G $ be a function satisfying the condition (A) and let $ \beta $ be defined as in Theorem \ref{thm4.1}. Then for each $ r,T>0 $ and $ m>\frac{1}{2} $, the following equation
	\begin{align}\label{XSDE}
		X_{t}=r+\beta_{t}+m\int_{0}^{t}\frac{1}{X_{s}}\mathrm{d}\langle \beta\rangle_{s}, \quad t\in [0,T],
	\end{align}
	has a solution $ X\in \tilde{M}_{G}^{2}(0,T) $ and $ X>0 $ q.s. Moreover, equation (\ref{XSDE}) has a unique solution $ X\in M_{G}^{2}(0,T) $.
\end{lemma}

\begin{proof} 
	For any integer $ n\geq1 $ and $ x\in\mathbb{R} $, define
	\[f_{n}(x):=\frac{1}{|x|}\wedge n.\]
	It is easy to check that $ f_{n}\in C_{b,Lip}(\mathbb{R}) $. Then we consider the following equation: 
	\begin{align}\label{xsden}
		X_{t}^{n}=r+\beta_{t}+m\int_{0}^{t}f_{n}(X_{s}^{n})\mathrm{d}\langle \beta\rangle_{s},  \quad t\in[0,T].
	\end{align}
	From Theorem \ref{thmsde}, it follows that equation (\ref{xsden}) has a unique solution $ X^{n}\in \tilde{M}_{G}^{2}(0,T) $. 
	
	Now we define the stopping time 
	\[\tau_{n}:=\inf\{t\geq0:X_{t}^{n}\leq\frac{1}{n} \}.\] 
	Similarly, we can define $ \tau_{n+1} $ with respect to $ X^{n+1} $. For $ t\leq \tau_{n}\wedge\tau_{n+1} $, 
	\begin{align}
		X_{t}^{n+1}-X_{t}^{n}={}&m\int_{0}^{t}(f_{n+1}(X_{s}^{n+1})-f_{n}(X_{s}^{n}))\mathrm{d}\langle \beta\rangle_{s}\notag\\
		={}&m\int_{0}^{t}(\frac{1}{X_{s}^{n+1}}-\frac{1}{X_{s}^{n}})\mathrm{d}\langle \beta\rangle_{s}.\notag
	\end{align}
	Then from the definition of $ \tau_{n} $, we get 
	\begin{align}
		|X_{t}^{n+1}-X_{t}^{n}|\leq{}&n(n+1)m\bar{\sigma}^{2}\int_{0}^{t}|X_{s}^{n+1}-X_{s}^{n}|\mathrm{d}s.\notag
	\end{align}
	Therefore, by Gronwall’s inequality, $ |X_{t}^{n+1}-X_{t}^{n}|=0 $ which implies $ X_{t}^{n+1}=X_{t}^{n} $ for every $ t\in[0,\tau_{n}\wedge\tau_{n+1}] $ q.s. Thus we infer that $ \tau_{n}\leq\tau_{n+1} $. 
	
	By Theorem \ref{quad B}, we obtain that $ \mathrm{d}\langle\beta\rangle_{t}=\theta_{t}^{2}\mathrm{d}t $ q.s. and $ \theta_{t}\in[\underline{\sigma},\bar{\sigma}] $. Note that $ \beta $ is a martingale under each $ P\in\mathcal{P} $, where $ \mathcal{P} $ is given in Theorem \ref{thm2.1}. Then for each $ P\in\mathcal{P} $ and the corresponding probability space $ (\Omega,(\mathcal{F}_{t})_{t\geq0},P) $, $ W_{t}^{P}:=\int_{0}^{t}\theta_{s}^{-1}\mathrm{d}\beta_{s},\ t\geq0, $ is a classical $ \mathcal{F}_{t} $-Brownian motion. Thus we have 
	\begin{align}
		X_{t}^{n}=r+\int_{0}^{t}\theta_{s}\mathrm{d}W_{s}^{P}+m\int_{0}^{t}f_{n}(X_{s}^{n})\theta_{s}^{2}\mathrm{d}s, \quad P\text{-}a.s.
	\end{align}
	Consequently, we infer that under each $ P\in\mathcal{P} $, $ X^{n} $ is a semimartingale. It is clear that $ \underline{\sigma}^{2}\leq\frac{\mathrm{d}\langle X^{n}\rangle_{t}}{\mathrm{d}t}\leq\bar{\sigma}^{2} $, $ P $-a.s. By Proposition 4.11 of \cite{liu2020exit}, we obtain that $ I_{[0,\tau_{n}\wedge T]} \in \tilde{M}_{G}^{2}(0,T) $. Therefore, 
	\begin{align}\label{eq4.7}
		X_{t\wedge\tau_{n}}^{n}={}&r+\beta_{t\wedge\tau_{n}}+m\int_{0}^{t\wedge\tau_{n}}f_{n}(X_{s}^{n})\mathrm{d}\langle \beta\rangle_{s}\notag\\
		={}&r+\int_{0}^{t}I_{[0,\tau_{n}]}(s)\mathrm{d}\beta_{s}+m\int_{0}^{t}\frac{1}{X_{s}^{n}}I_{[0,\tau_{n}]}(s)\mathrm{d}\langle \beta\rangle_{s},  \quad t\in[0,T].
	\end{align}	

	Another step is to show that $ \tau_{n}\rightarrow\infty $ as $ n\rightarrow\infty $. For each $ x>0 $, define $ h(x)=x^{1-2m} $. 
	Applying the It\^{o} formula to $ h(X_{t\wedge\tau_{n}}^{n}) $, we get
	\[h(X_{t\wedge\tau_{n}}^{n})=r^{1-2m}+\int_{0}^{t}(1-2m)(X_{s}^{n})^{-2m}I_{[0,\tau_{n}]}(s)\mathrm{d}\beta_{s},\quad t\in[0,T].\] 
	This implies that $ h(X_{t\wedge\tau_{n}}^{n}),\ t\in[0,T], $ is a symmetric martingale. Then for every $ t\in[0,T] $, 
	\begin{align*}
		\hat{\mathbb{E}}[h(X_{t\wedge\tau_{n}}^{n})]
		={}&\hat{\mathbb{E}}[h(X_{t}^{n})I_{\{t<\tau_{n}\}}+h(X_{\tau_{n}}^{n})I_{\{\tau_{n}\leq t\}}]\\
		\geq{}&\hat{\mathbb{E}}[h(X_{\tau_{n}}^{n})I_{\{\tau_{n}\leq t\}}]\\
		={}&n^{2m-1}\mathrm{c}(\{\tau_{n}\leq t\}).
	\end{align*} 
    Thus we have 
	\begin{align*}
		r^{1-2m}=\hat{\mathbb{E}}[h(X_{0}^{n})]=\hat{\mathbb{E}}[h(X_{t\wedge\tau_{n}}^{n})]\geq n^{2m-1}\mathrm{c}(\{\tau_{n}\leq t\}).
	\end{align*}
	Therefore, we obtain that
	\begin{align}\label{tau}
		\mathrm{c}(\{\tau_{n}\leq t\})\leq (nr)^{1-2m}\rightarrow0 \quad \text{as } n\rightarrow\infty .
	\end{align}
	Due to (\ref{tau}) and $ \tau_{n}\uparrow $, it follows that $ \tau_{n}\rightarrow\infty $ q.s. as $ n\rightarrow\infty $. 
	
	Now we prove that (\ref{XSDE}) exists a solution which belongs to $ \tilde{M}_{G}^{2}(0,T) $. Let us define $ X=\lim\limits_{n\rightarrow\infty}X^{n} $ q.s. Notice that $ \lim\limits_{n\rightarrow\infty}\tau_{n}=\infty $ q.s. and $ X_{t}=X_{t}^{n}>0 $ for every $ t\in[0,\tau_{n}] $ q.s. Then we obtain $ X>0 $ q.s. By letting $ n \rightarrow\infty $ in equation (\ref{xsden}), we obtain that $ X $ satisfies equation (\ref{XSDE}) in the quasi-surely sense. In the following, we only need to prove $ X\in\tilde{M}_{G}^{2}(0,T) $. Set 
	\[\tilde{X}_{t}^{n}:=X_{t}^{n}I_{[0,\tau_{n}\wedge T]}(t).\]
	Clearly, $ \tilde{X}^{n}\in \tilde{M}_{G}^{2}(0,T) $ and $ X_{t}=\tilde{X}^{n}_{t} $ for every $ t\in[0,\tau_{n}] $ q.s. Thus, we have
	\begin{align}
		\hat{\mathbb{E}}\left[ \int_{0}^{T}\left| X_{t}-\tilde{X}_{t}^{n}\right| ^{2}\mathrm{d}t\right]={}&\hat{\mathbb{E}}\left[\int_{0}^{T}\left| (X_{t}-\tilde{X}_{t}^{n})I_{\{T<\tau_{n}\}}+(X_{t}-\tilde{X}_{t}^{n})I_{\{\tau_{n}\leq T\}}\right| ^{2}\mathrm{d}t \right]    \notag\\
		={}&\hat{\mathbb{E}}\left[\int_{0}^{T}\left| (X_{t}I_{[0,\tau_{n})}(t)+X_{t}I_{[\tau_{n},T]}(t)-\tilde{X}_{t}^{n})I_{\{\tau_{n}\leq T\}}\right| ^{2} \mathrm{d}t \right] \notag\\
		\leq{}&\int_{0}^{T}\hat{\mathbb{E}}\left[(X_{t})^{2}I_{\{\tau_{n}\leq T\}} \right]\mathrm{d}t.\notag    	
	\end{align}
	By applying the It\^{o} formula to $ (X_{t}^{n})^{2} $ and $ (X_{t}^{n})^{4} $, we obtain that
	\begin{align*}
		(X_{t}^{n})^{2}={}&r^{2}+\int_{0}^{t}2X_{s}^{n}\mathrm{d}\beta_{s}+\int_{0}^{t}2mX_{s}^{n}f_{n}(X_{s}^{n})\mathrm{d}\langle \beta\rangle_{s}+\langle \beta\rangle_{t}\\
		\leq{}&r^{2}+\int_{0}^{t}2X_{s}^{n}\mathrm{d}\beta_{s}+(2m+1)\langle \beta\rangle_{t},\notag\\
		(X_{t}^{n})^{4}={}&r^{4}+\int_{0}^{t}4(X_{s}^{n})^{3}\mathrm{d}\beta_{s}+\int_{0}^{t}4m(X_{s}^{n})^{3}f_{n}(X_{s}^{n})\mathrm{d}\langle \beta\rangle_{s}+\int_{0}^{t}6(X_{s}^{n})^{2}\mathrm{d}\langle \beta\rangle_{s}\\
		\leq{}&r^{4}+\int_{0}^{t}4(X_{s}^{n})^{3}\mathrm{d}\beta_{s}+\int_{0}^{t}(4m+6)(X_{s}^{n})^{2}\mathrm{d}\langle \beta\rangle_{s}.
	\end{align*}
	Therefore, 
	\begin{align*}
		\hat{\mathbb{E}}[(X_{t}^{n})^{2}]\leq{}& r^{2}+(2m+1)\bar{\sigma}^{2}T,\\
		\hat{\mathbb{E}}[(X_{t}^{n})^{4}]\leq{}&r^{4}+(4m+6)\bar{\sigma}^{2}\int_{0}^{t}\hat{\mathbb{E}}[(X_{s}^{n})^{2}]\mathrm{d}s\\
		\leq{}&C_{r,m,\bar{\sigma},T}.
	\end{align*}
    where $ C_{r,m,\bar{\sigma},T}=r^{4}+(4m+6)\bar{\sigma}^{2}r^{2}+(2m+1)(4m+6)\bar{\sigma}^{4}T^{2} $.
	Note that $ (X_{t}^{n})^{4}\geq0 $. Then by Fatou's Lemma, we obtain that for each $P\in\mathcal{P}  $, 
	\[ E_{P}[(X_{t})^{4}]=E_{P}[\liminf_{n\rightarrow\infty}(X_{t}^{n})^{4}]\leq \liminf_{n\rightarrow\infty}E_{P}[(X_{t}^{n})^{4}]\leq \liminf_{n\rightarrow\infty}\hat{\mathbb{E}}[(X_{t}^{n})^{4}] ,\]
	which implies
	\begin{align}\label{4.9}
		\hat{\mathbb{E}}[(X_{t})^{4}]\leq\liminf_{n\rightarrow\infty}\hat{\mathbb{E}}[(X_{t}^{n})^{4}]\leq C_{r,m,\bar{\sigma},T} .
	\end{align}
	By (\ref{4.9}) and H\"{o}lder's inequality, we conclude that 
	\begin{align*}
		\hat{\mathbb{E}}\left[ \int_{0}^{T}|X_{t}-\tilde{X}_{t}^{n}|^{2}\mathrm{d}t\right]\leq{}&\int_{0}^{T}\hat{\mathbb{E}}\left[(X_{t})^{2}I_{\{\tau_{n}\leq T\}} \right] \mathrm{d}t\\
		\leq{}&\int_{0}^{T}\left(\hat{\mathbb{E}}[(X_{t})^{4}] \right)^{\frac{1}{2}} \left( \hat{\mathbb{E}}[I_{\{\tau_{n}\leq T\}}]\right)^{\frac{1}{2}}\mathrm{d}t\\
		\leq{}&\int_{0}^{T}(C_{r,m,\bar{\sigma},T})^{\frac{1}{2}}(\mathrm{c}(\{\tau_{n}\leq T\}))^{\frac{1}{2}}\mathrm{d}t\\
		\leq{}&(C_{r,m,\bar{\sigma},T})^{\frac{1}{2}}T(nr)^{\frac{1}{2}-m}\rightarrow 0 \quad \text{as } n\rightarrow\infty.
	\end{align*}
	Then we obtain $ X\in \tilde{M}_{G}^{2}(0,T) $.
	
	Finally, we need to show the uniqueness of the solution. Let $ X'\in M_{G}^{2}(0,T) $ be a solution of equation (\ref{XSDE}). Similarly, we can define $ \tau'_{n}:=\inf\{t\geq0:X'_{t}\leq\frac{1}{n}\} $. Then we obtain
	\begin{align*}
		X_{t\wedge\tau_{n}\wedge\tau_{n}^{'}}=r+{}&\beta_{t\wedge\tau_{n}\wedge\tau_{n}^{'}}+m\int_{0}^{t\wedge\tau_{n}\wedge\tau_{n}^{'}}\frac{1}{X_{s}}\mathrm{d}\langle\beta\rangle_{s}, \\
		X_{t\wedge\tau_{n}\wedge\tau_{n}^{'}}^{'}=r+{}&\beta_{t\wedge\tau_{n}\wedge\tau_{n}^{'}}+m\int_{0}^{t\wedge\tau_{n}\wedge\tau_{n}^{'}}\frac{1}{X_{s}^{'}}\mathrm{d}\langle\beta\rangle_{s}.
	\end{align*}
    For $ t\in[0,\tau_{n}\wedge\tau_{n}^{'}] $, 
    \begin{align*}
    	|X_{t}-X_{t}^{'}|\leq{}&n^{2}m\bar{\sigma}^{2}\int_{0}^{t}|X_{s}-X_{s}^{'}|\mathrm{d}s.
    \end{align*}
    By Gronwall’s inequality, we have $ X_{t}=X'_{t} $ for every $ t\in[0,\tau_{n}\wedge\tau'_{n}] $ q.s. It follows that $ \tau_{n}=\tau'_{n} $. Note that $ \tau_{n}\rightarrow\infty $ q.s. as $ n\rightarrow\infty $. Then we obtain that $ X=X' $ q.s. The proof is complete.
\end{proof}

\begin{proof}[Proof of Theorem \ref{thm4.1}] 
	Since the Euclidean norm of $ x\in\mathbb{R}^{d} $ is not $ C^{2}(\mathbb{R}^{d}) $ at the origin, we cannot use the It\^{o} formula to $ R_{t}=|B_{t}^{x}| $ directly. Based on the above argument, we consider 
	\[Y_{t}:=R_{t}^{2}=|B_{t}^{x}|^{2}=(\bar{B}_{t}^{(1)})^{2}+\cdots+(\bar{B}_{t}^{(d)})^{2} .\]
	We use the It\^{o} formula to $ Y_{t} $. According to Lemma \ref{lem4.1}, we get
	\[ Y_{t}=r^{2}+2\sum_{i=1}^{d}\int_{0}^{t}\bar{B}_{s}^{(i)}\mathrm{d}B_{s}^{(i)}+d\langle B^{(1)}\rangle_{t}. \]
	For every $ n\in\mathbb{N} $ and $ \epsilon_{n}=2^{-n} $, we set
	\begin{equation*}
		\varphi_{\epsilon_{n}}(y)=
		\begin{cases}
			\frac{3}{8}\sqrt{\epsilon_{n}}+\frac{3}{4\sqrt{\epsilon_{n}}}y-\frac{1}{8\epsilon_{n}\sqrt{\epsilon_{n}}}y^{2}; \quad &y<\epsilon_{n},\\
			\sqrt{y}; \quad &y\geq\epsilon_{n}. 
		\end{cases}
	\end{equation*}
	Then $ \varphi_{\epsilon_{n}}\in C^{2}(\mathbb{R}) $ and $ \lim_{n\rightarrow\infty}\varphi_{\epsilon_{n}}(y)=\sqrt{y} $ for $ y\geq0 $. Applying the Itô formula to $ \varphi_{\epsilon_{n}}(Y_{t}) $, we obtain that
	\begin{align}\label{ep}
		\varphi_{\epsilon_{n}}(Y_{t})=\varphi_{\epsilon_{n}}(r^{2})+\sum_{i=1}^{d}I_{t}^{(i)}(\epsilon_{n})+J_{t}(\epsilon_{n})+K_{t}(\epsilon_{n}),
	\end{align}
	where 
	\begin{align*}
		&I_{t}^{(i)}(\epsilon_{n}):=\int_{0}^{t}\left[I_{\{Y_{s}<\epsilon_{n}\}}\frac{1}{2\sqrt{\epsilon_{n}}}(3-\frac{Y_{s}}{\epsilon_{n}})+I_{\{Y_{s}\geq\epsilon_{n}\}}\frac{1}{R_{s}} \right]\bar{B}_{s}^{(i)}\mathrm{d}B_{s}^{(i)},\\ 
		&J_{t}(\epsilon_{n}):=\int_{0}^{t}I_{\{Y_{s}\geq\epsilon_{n}\}}\frac{d-1}{2R_{s}}\mathrm{d}\langle B^{(1)}\rangle_{s},\\
		&K_{t}(\epsilon_{n}):=\int_{0}^{t}I_{\{Y_{s}<\epsilon_{n}\}}\frac{1}{4\sqrt{\epsilon_{n}}}\left(3d-(d+2)\frac{Y_{s}}{\epsilon_{n}}\right)\mathrm{d}\langle B^{(1)}\rangle_{s}.
	\end{align*}

	Now we show that (\ref{ep}) tends to (\ref{R}) as $ n\rightarrow\infty $.
	First we prove that $ K_{t}(\epsilon_{n})\rightarrow0 $ as $ n\rightarrow \infty$. 
	\begin{align*}
		\hat{\mathbb{E}}\left[|K_{t}(\epsilon_{n})| \right] \leq{}&\hat{\mathbb{E}}\left[\int_{0}^{t}\frac{3d}{4\sqrt{\epsilon_{n}}}I_{\{Y_{s}<\epsilon_{n}\}}\mathrm{d}\langle B^{(1)}\rangle_{s}	 \right] \\
		\leq{}& \frac{3\bar{\sigma}^{2}d}{4}\hat{\mathbb{E}}\left[\int_{0}^{t}\frac{1}{\sqrt{\epsilon_{n}}}I_{\{Y_{s}<\epsilon_{n}\}}\mathrm{d}s \right].  
	\end{align*}
	From Lemma \ref{im}, we conclude that
	\begin{align*}
		\hat{\mathbb{E}}\left[|K_{t}(\epsilon_{n})| \right]\leq\frac{3\bar{\sigma}^{2}d}{4}\hat{\mathbb{E}}\left[\int_{0}^{t}\frac{1}{\sqrt{\epsilon_{n}}}I_{\{|B_{s}^{x}|<\sqrt{\epsilon_{n}}\}}\mathrm{d}s \right] \rightarrow 0  \quad \text{as } n\rightarrow \infty.
	\end{align*}
	Then we show that $ \sum_{i=1}^{d}I_{t}^{(i)}(\epsilon_{n})\rightarrow\beta_{t} $ as $ n\rightarrow\infty $. Let us define $ \beta_{t}^{(i)}=\int_{0}^{t}\frac{\bar{B}_{s}^{(i)}}{R_{s}}\mathrm{d}B_{s}^{(i)}$ for $ 1\leq i\leq d $. Using Lemma \ref{im} again, we get  
	\begin{align*}
		\hat{\mathbb{E}}\left[\left| I_{t}^{(i)}(\epsilon_{n})-\beta_{t}^{(i)}\right| ^{2} \right]={}&\hat{\mathbb{E}}\left[\left| \int_{0}^{t}I_{\{Y_{s}<\epsilon_{n}\}}\left(\frac{1}{2\sqrt{\epsilon_{n}}}(3-\frac{Y_{s}}{\epsilon_{n}})-\frac{1}{R_{s}}\right)\bar{B}_{s}^{(i)}\mathrm{d}B_{s}^{(i)}\right| ^{2} \right]  \\
		={}&\hat{\mathbb{E}}\left[\int_{0}^{t}\left|I_{\{Y_{s}<\epsilon_{n}\}}\left(1-\frac{1}{2}\sqrt{\frac{Y_{s}}{\epsilon_{n}}}(3-\frac{Y_{s}}{\epsilon_{n}})\right)^{2}(\frac{\bar{B}_{s}^{(i)}}{R_{s}})^{2} \right|^{2}\mathrm{d}\langle B^{(1)}\rangle_{s}  \right]\\
		\leq{}&\bar{\sigma}^{2}\hat{\mathbb{E}}\left[\int_{0}^{t}I_{\{Y_{s}<\epsilon_{n}\}}\mathrm{d}s \right]\rightarrow 0 \quad \text{as } n\rightarrow\infty.  
	\end{align*} 
	Therefore, $\sum_{i=1}^{d}I_{t}^{(i)}(\epsilon_{n})\rightarrow \beta_{t}  $ as $ n\rightarrow\infty $. Finally, it remains to show that $J_{t}(\epsilon_{n}) \rightarrow \int_{0}^{t}\frac{d-1}{2R_{s}}\mathrm{d}\langle \beta\rangle_{s} $ as $ n\rightarrow\infty $. Clearly $g_{n}:=I_{\{Y_{t}\geq\epsilon_{n}\}}\frac{d-1}{2R_{t}},\ t\in[0,T],$ belongs to $ M_{G}^{1}(0,T)  $. If we can prove $ g_{n} $ is a Cauchy sequence under the norm $ \|\cdot\|_{M_{G}^{1}} $, then the limit $ \frac{d-1}{2R_{t}},\ t\in[0,T], $ belongs to $ M_{G}^{1}(0,T) $ which implies $ J_{t}(\epsilon_{n}) \rightarrow \int_{0}^{t}\frac{d-1}{2R_{s}}\mathrm{d}\langle \beta\rangle_{s} $ as $ n\rightarrow\infty $.
	For any $ n,m\in\mathbb{N} $, we have 
	\begin{align*}
		\left\|g_{n+m}-g_{n} \right\|_{M_{G}^{1}}\leq{}&\left\|g_{n+m}-g_{n+m-1} \right\|_{M_{G}^{1}}+\cdots+\left\|g_{n+1}-g_{n} \right\|_{M_{G}^{1}}\\
		={}&\sum_{j=1}^{m}\left\|I_{\{\epsilon_{n+j-1}\geq R_{t}\geq\epsilon_{n+j}\}}\frac{d-1}{2R_{t}} \right\|_{M_{G}^{1}}\\
		={}&\sum_{j=1}^{m}\hat{\mathbb{E}}\left[\int_{0}^{T}\left|I_{\{\epsilon_{n+j-1}\geq R_{t}\geq\epsilon_{n+j}\}}\frac{d-1}{2R_{t}} \right| \mathrm{d}t \right]\\
		\leq{}&\sum_{j=1}^{m}\frac{d-1}{2}\hat{\mathbb{E}}\left[\int_{0}^{T}I_{\{|B_{t}^{x}|\leq\epsilon_{n+j-1}\}}2^{n+j}\mathrm{d}t \right].   
	\end{align*}
	By Lemma \ref{im}, there exists $ \alpha\in(\frac{1}{2},1) $ such that
	\begin{align*}
		\hat{\mathbb{E}}\left[\int_{0}^{T}I_{\{|B_{t}^{x}|\leq\epsilon_{n+j-1}\}}2^{n+j}\mathrm{d}t \right]\leq{}&\int_{0}^{T}\hat{\mathbb{E}}[I_{\{|B_{t}^{x}|\leq\epsilon_{n+j-1}\}}]2^{n+j}\mathrm{d}t\\
		\leq{}&\int_{0}^{T}\exp(\frac{1}{2\bar{\sigma}^{2}})\frac{\epsilon_{n+j-1}^{2\alpha}}{t^{\alpha}}2^{n+j}\mathrm{d}t\\
		={}&\frac{2}{1-\alpha}\exp(\frac{1}{2\bar{\sigma}^{2}})T^{1-\alpha}(2^{n+j-1})^{1-2\alpha}.
	\end{align*}
	Since the series $ \sum_{j=1}^{\infty}2^{(1-2\alpha)j} $ is convergent, we get 
	\begin{align*}
		\left\|g_{n+m}-g_{n} \right\|_{M_{G}^{1}}\leq{}&\sum_{j=1}^{m}\frac{d-1}{1-\alpha}\exp(\frac{1}{2\bar{\sigma}^{2}})T^{1-\alpha}(2^{n+j-1})^{1-2\alpha}\rightarrow 0 \quad \text{as } n,m\rightarrow\infty.
	\end{align*}
	Now we obtain $ R $ satisfies (\ref{R}). 
	
	If $ d\geq3, \ r>0 $ in (\ref{R}), then $ \frac{d-1}{2}\geq1>\frac{1}{2} $. According to the Lemma \ref{lem4.6}, we obtain that $ R $ is the unique solution which belongs to $ M_{G}^{2}(0,T) $. Moreover, $ R\in\tilde{M}_{G}^{2}(0,T) $ and $ R>0 $ q.s.
\end{proof}

\begin{remark}
	If $ \underline{\sigma}=\bar{\sigma}=1 $, then $ R $ is reduced to the Bessel process. For $ d\geq2 $ and $ r>0 $, the Bessel process is the unique solution of the following equation 
	\begin{align}\label{CR}
		\begin{cases}
			\mathrm{d}R_{t}=\frac{d-1}{2R_{t}}\mathrm{d}t+	\mathrm{d}\beta_{t},\\
			R_{0}=r,
		\end{cases}
	\end{align} 
	and never reaches the origin. According to Theorem \ref{thm4.1}, the above properties hold with $ d\geq3 $. The proof of the uniqueness of the solution depends on the Yamada–Watanabe uniqueness criterion, but this does not have a more general result under the framework of $ G $-expectation. Briefly speaking, the proof of the Bessel process which never reaches the origin depends critically on some manipulation of stopping times. However, since there are no corresponding results regarding the stopping time $ \tau_{n} $ under the framework of $ G $-expectation, the above approach is not currently applicable. Therefore, we leave the discussion for future research topics.
\end{remark}

\section*{Acknowledgements}

The authors thank the anonymous referees for their very careful reading and many valuable suggestions.

\section*{Declarations}

\subsection*{Funding}
This work is supported by the National Natural Science Foundation of China (Grant No. 12326603, 11671231) and the National Key R\&D Program of China (Grant No. 2018YFA0703900).
\subsection*{Ethical approval}
Not applicable.
\subsection*{Informed consent}
Not applicable.
\subsection*{Author Contributions}
All authors contributed equally to each part of this work. All authors read and approved the final manuscript.
\subsection*{Data Availability Statement}
Not applicable.
\subsection*{Conflict of Interest}
The authors declare that they have no known competing financial interests or personal relationships that could have appeared to influence the work reported in this paper.
\subsection*{Clinical Trial Number}
Not applicable.

\end{document}